\documentclass[preprint,12pt]{elsarticle1}
\usepackage{subfig}
\usepackage{graphicx}
\usepackage{caption,color}   
\usepackage{CJK} 
\usepackage{amssymb}
\usepackage{amsthm}
\usepackage{amsmath}
\usepackage{epic}
\usepackage{setspace}
\usepackage{float}
\usepackage{multirow}
\usepackage{hyperref}
\newtheorem{thm}{Theorem}[section]
\newtheorem{cor}[thm]{Corollary}

\newtheorem{lem}[thm]{Lemma}

\newtheorem{defi}[thm]{Definition}
\newtheorem{claim}{Claim}

\numberwithin{equation}{section}
\journal{}
\begin{document}
\begin{spacing}{1.15}
\begin{frontmatter}
\title{\textbf{~\\The high order spectral extrema of $2K_r$-free graphs}}

\author[]{Changjiang Bu\corref{mycorrespondingauthor}}
\ead{buchangjiang@hrbeu.edu.cn}
\author[]{Yifan Sun}
\author[]{Haotian Zeng}
\address{School of Mathematical Sciences, Harbin Engineering University, Harbin 150001, PR China}
\cortext[mycorrespondingauthor]{Corresponding author}


\begin{abstract}
\indent 
   In this paper, we determine the graphs with maximum value of the sum number from $k$-clique spectral radius  to $(2r-1)$-clique spectral radius among all $2K_{r}$-free graphs on $n$ vertices for $ r\le k$ and large $n$. We also determine the graphs with maximum $3$-clique spectral radius among all $2K_{3}$-free graphs on $n$ vertices.  Our results are spectral versions of some results on generalized Turán numbers.
\end{abstract}

\end{frontmatter}

\section{Introduction}
\label{Introduction }
\indent The graphs considered in this paper are undirected and simple. 
For a graph  $G=(V(G),E(G))$,  let $N_{G}(i_1,\cdots,i_{s})= N_{G}(i_1)\cap N_{G}(i_2)\cap \cdots \cap N_{G}(i_s) $ where $N_G(i_j)$ is  the neighborhood of the vertex $i_j\in V(G)$ 
for $j=1,\cdots,s$. For $U\subseteq V(G)$, let $G[U]$ denote the induced subgraph  of $G$ by $U$.
 An \emph{$r$-clique} of $G$ is a subset of $V(G)$ such that the induced subgraph of this subset is  a complete graph with $r$ vertices.
Let $kG$  be the disjoint union of $k$ copies of $G$. Let $G \vee H$ be the graph which is obtained by  joining every pair $v_{1}$, $v_{2}$ of vertices with ${{v}_{1}}\in V\left( G \right)$, ${{v}_{2}}\in V\left( H \right)$.

\indent For two graphs $G$ and $F$, $G$ is called $F$-\emph{free} if $G$ does not contain  $F$ as a subgraph. The Tur\'{a}n number $ex(n,F)$ is the maximum number of edges of an $F$-free graph on $n$ vertices. Let $T_r(n)$ be the complete $r$-partite graph on $n$ vertices without two part sizes differing by more than one. 
In 1941, Tur\'{a}n \cite{turan1941} gave the following classical theorem.

\begin{thm}\emph{\cite{turan1941}}
 Let $G$ be a  $K_{r+1}$-free graph on $n$ vertices with $ex(n,K_{r+1})$ edges. Then $G\cong T_r(n)$.
\end{thm}

For two graphs $H$ and $F$, the generalized Tur\'{a}n number is the maximum number of copies in subgraph $H$ of an $F$-free graph on $n$ vertices, 
denoted by $ex(n,H,F)$. Zykov \cite{Zykov1949} and  Erd\H{o}s \cite{Erdos1962}  determined the graphs with maximum number of copies of $K_s$ among all $K_{r+1}$-free graphs for $2\le s \le r$. In 2016, Alon and Shikhelman \cite{Alon} systematically studied the function $ex(n, H, F)$. Other research on the generalized Tur\'{a}n numbers can be referred to \cite{Fang, Gerbner1, Luo, Ma}.

  In 2022, Yuan and Yang {\cite{yuan20221}}  determined the graphs with the maximum number of copies of $K_r$ among all $2K_r$-free graphs on $n$ vertices for $n\ge 3r^5$. They also determined the graphs with the maximum number of copies of $K_3$ among all $2K_3$-free graphs on $n$ vertices for all $n$. 
\begin{thm}\label{yang}\emph{\cite{yuan20221}}
 Let $G$ be a  $2K_{r}$-free graph on $n$ vertices with $ex(n,K_{r},2K_{r})$ copies of $K_r$, where $r\ge 3$. If $n\ge 3r^5$, then $G\cong K_1\vee T_{r-1}(n-1)$.
\end{thm}
\begin{thm}\label{yang1}\emph{\cite{yuan20221}}
 Let $n\ge 7$. Then
  \[ex(n,K_3,2K_3)=max\left\{{3n-8, \left \lfloor \frac{(n-1)^2 }{4} 
  \right \rfloor   }\right\}.\]
  and the extremal graphs are $K_3\vee (n-3)K_1$ and $K_1\vee T_2(n-1)$, respectively.
\end{thm}

Let $ex(n,\{K_k, K_{k+1},\allowbreak\dots,\allowbreak,K_{2r-1}\},2K_r)$ be the maximum value of the sum from the number of $k$-cliques to the number of $(2r-1)$-cliques  among all $2K_r$-free graphs on $n$ vertices. Let \( C_{r}(G) \) be the set of all  $r$-cliques in  $G$. In 2024, Gerbner and  Patk\'{o}s \cite{Gerbner} determined the graphs with  the maximum  number of $k$-cliques among all $2K_r$-free graphs.

\begin{thm}\label{Gerbner}\emph{\cite{Gerbner}}
Let $r\geq2$ and $n$ large enough. \\
$(i)$ If $k<r$, then 
\[ex(n, K_k,2K_r) =|C_k(K_1 \vee T_{r-1}(n-1))|.\]
$(ii)$ If $r \le k \le2r-1$, then 
\[ex(n,\{K_k, K_{k+1},...,K_{2r-1}\},2K_r) = |C_k(K_{2k-2r+1} \vee T_{2r-k-1}(n-2k+2r-1))|.\]
\end{thm}

In 2007, Nikiforov \cite{nikiforov2007} determined the graphs on $n$ vertices with maximum spectral radius among all \( K_{r+1} \)-free graphs. In 2023, Ni, Wang and Kang \cite{Ni2023} determined the graphs on $n$ vertices with maximum spectral radius among all $kK_{r+1}$-free graph for $k\ge2$, $r\ge2$.

In 2023, Liu and Bu \cite{liu2023mantel} defined the \( r \)-clique tensor of a graph.
\begin{defi}\label{wei1}
\emph{\cite{liu2023mantel}}  For a graph $G$ with $n$ vertices and an integer $r\ge 2$, the $r$-clique tensor $\mathcal{A}_r(G)=\left( {{a}_{{{i}_{1}}{{i}_{2}}\ldots {{i}_{r}}}} \right)$ is an $r$-order $n$-dimensional tensor, with entries
\[
a_{i_1 i_2 \ldots i_r} = 
\begin{cases}
   \dfrac{1}{(r-1)!}, & \{ i_1, i_2, \dots, i_r \} \in C_r(G), \\[8pt]
   0, & \text{otherwise}.
\end{cases}
\]
The $r$-clique spectral radius of $G$ is the spectral radius of $\mathcal{A}_r(G)$, denoted by $\mu_r(G)$. If $r=2$, $\mu_r(G)$ is the spectral radius of $G$.
\end{defi}

 Liu and Bu \cite{liu2023mantel}  generalized the spectral Mantel Theorem \cite{nosal1970}, determined the maximum \( r \)-clique spectral radius of  a \( K_{r+1} \)-free graph. In 2025, Yu and Peng \cite{Peng2025} gave a clique spectral version of the result by Zykov \cite{Zykov1949} and  Erd\H{o}s \cite{Erdos1962}, determined the graphs with maximum \( s \)-clique spectral radius of  a \( K_{r+1} \)-free graph for $2\le s\le r$. In 2026, Liu and Bu \cite{Li2026} obtained a lower bound of the number of $r$-cliques in $G$ via the  $r$-clique spectral radius.
\begin{thm}\label{Liu}\emph{\cite{Li2026}}
Let $G$ be a graph with the size of the largest clique \( \omega \). Then 

 \[\mu_r(G)\le \frac{r}{\omega}\tbinom{\omega}{r}^{\frac{1}{r}}|C_r(G)|^{\frac{r-1}{r}},\]
where $\omega\ge r\ge 2$. Moreover, if $G$ is a complete regular $\omega$-partite graph,  then the equality is achieved in the above inequality.
\end{thm}
Other research on high order spectra extremal results of graphs, please refer to \cite{Wanga,1,peng2026,a}. 
 
 In this paper, we establish the spectral versions of Theorem \ref{Gerbner} $(\romannumeral 2)$ and Theorem \ref{yang1}. Let $\mathrm{e}$ be the base of the natural logarithm.
\begin{thm}\label{sun2}
Let $G$ be a $2K_r$-free graph on $n$ vertices, where $r\ge3$. If $r\le k\le 2r-1$ and $n\ge (2r-k-1)\left(\mathrm{e}r^4k^{\frac{k}{2}}\right)^{\frac{k}{2k-2k+1}}+k $, then we have
\[\mu_k(G)+\mu_{k+1}(G)+\cdots+\mu_{2r-1}(G) \le  \mu_k( K_{2k-2r+1}\vee T_{2r-k-1}(n-2k+2r-1)).\]
$(\romannumeral 1)$ For $k=2r-1$, equality holds if and only if $|C_k(G)|=1$. \\
$(\romannumeral 2)$ For $r\le k\le 2r-2$, equality holds if and only if $G\cong  K_{2k-2r+1}\vee T_{2r-k-1}(n-2k+2r-1).$
\end{thm}

Since $K_{2k-2r+1}\vee T_{2r-k-1}(n-2k+2r-1)$ is a complete $k$-partite graph, we have $K_{2k-2r+1}\vee T_{2r-k-1}(n-2k+2r-1)$ is $K_{k+1}$-free, then  we have $\mu_i(G)=0$ for $i=k+1,\cdots,2r-1$. Thus, by Theorem \ref{sun2}, we obtain the following corollary.

\begin{cor}\label{sun2s}
    Let  $G$ be a $2K_r$-free graph on $n$ vertices, where $r\ge3$. If $r\le k\le 2r-1$ and $n\ge (2r-k-1)\left(\mathrm{e}r^4k^{\frac{k}{2}}\right)^{\frac{k}{2k-2k+1}}+k $, then we have
\[\mu_k(G) \le  \mu_k( K_{2k-2r+1}\vee T_{2r-k-1}(n-2k+2r-1)).\]
$(\romannumeral 1)$ For $k=2r-1$, equality holds if and only if $|C_k(G)|=1$. .\\
$(\romannumeral 2)$ For $r\le k\le 2r-2$, equality holds if and only if $G\cong  K_{2k-2r+1}\vee T_{2r-k-1}(n-2k+2r-1)$. 
\end{cor}

In this paper, we determine the graphs with maximum 3-clique radius spectral among all $2K_3$-free graphs on $n$ vertices.

\begin{thm}\label{cor1}

  Let  $G$ be a $2K_3$-free graph on $n$ vertices.  Then
\[
\mu_{3}(G) \le 
\begin{cases}
\mu_{3}(K_{n}), & 3 \le n \le 5, \\
\mu_{3}(K_{5} \cup K_1), & n = 6, \\
\mu_{3}(K_{3} \vee (n-3)K_1), & 7 \le n \le 13, \\
\mu_{3}(K_{1} \vee T_2(n-1)), & n \ge 14.
\end{cases}
\]
$(\romannumeral 1)$ For $3 \le n \le 5$, equality holds if and only if  $G\cong K_{n}$. \\
$(\romannumeral 2)$ For $n=6$, equality holds if and only if  $G\cong K_{5}\cup K_1$ or $G\cong G_0$, where $G_0$ is obtained by adding an edge between a vertex in $K_5$ and an isolated vertex.\\
$(\romannumeral 3)$ For $7\le n\le 13$, equality holds if and only if $G\cong K_{3}\vee (n-3)K_1$. \\
$(\romannumeral 4)$ For $n\ge 14$, equality holds if and only if $G\cong K_{1}\vee T_{2}(n-1)$.
\end{thm}

The remainder of this paper is organized as follows. In Section 2, some definitions and lemmas required for proofs are introduced. In Section 3, we give proofs of Theorems \ref{sun2} and  \ref{cor1}.

\section{Preliminaries}
\indent In this section, we introduce some definitions and lemmas required for the proofs.

\indent In 2005, Qi \cite{qi2005} and Lim \cite{lim2005} proposed tensor eigenvalues respectively. For an \( r \)-order \( n \)-dimensional complex tensor \( \mathcal{A} = \left( a_{i_1 i_2 \ldots i_r} \right) \) and an \( n \) 
-dimensional complex vector \( x = \left( x_1, x_2, \ldots, x_n \right)^{\top} \), \( \mathcal{A} x^{r-1} \) is an $n$-dimensional complex vector whose \( i \) component is  
\[(\mathcal{A}x^{r-1})_i=\sum_{i_2, \ldots, i_r = 1}^n a_{i i_2 \ldots i_r} x_{i_2} \cdots x_{i_r},\]  
If there exist a complex number \( \lambda \) and a nonzero complex vector \( x = \left( x_1, x_2, \ldots, x_n \right)^{\top} \) satisfying  
\[\mathcal{A} x^{r-1} = \lambda x^{[r-1]},\]  
then \( \lambda \) is called an eigenvalue of \( \mathcal{A} \), and \( x \) is the eigenvector of \( \mathcal{A} \) corresponding to \( \lambda \), where \( x^{[r-1]} = \left( x_1^{r-1}, \ldots, x_n^{r-1} \right)^{\top} \).  
The set of all eigenvalues of \( \mathcal{A} \) is denoted by \( \Lambda(\mathcal{A}) \), and the spectral radius is defined as \( \mu(\mathcal{A}) = \max \left\{ |\mu| : \mu \in \Lambda(\mathcal{A}) \right\} \).

We call a real tensor \emph{nonnegative} if each element  is nonnegative. For an \( r \)-order \( n \)-dimensional tensor \( \mathcal{A} = \left( a_{i_1 i_2 \ldots i_r} \right) \). For any permutation $\sigma$ of $i_1 i_2 \ldots i_r$, if $a_{i_1 i_2 \ldots i_r}=a_{\sigma (i_1 i_2 \ldots i_r)}$ for $i_j\in[n]$ and $j=1,2,\cdots,r$, then $\mathcal{A}$ is called \emph{symmetric}.
\begin{lem}\label{lem1}
\emph{\cite{qi2013}} Let $\mathcal{A}$ be an \( r \)-order \( n \)-dimensional symmetric nonnegative tensor. Then 
$$\mu \left( \mathcal{A}\right)=\max \left\{ {{x}^{\top}}\mathcal{A}{{x}^{r-1}}:\sum\limits_{i=1}^{n}{x_{i}^{r}=1, \left( x_1, x_2, \ldots, x_n \right)^{\top}\in\mathbb{R}_{+}^{n} } \right\},$$
where $\mathbb{R}_{+}^{n}$ denotes the set of all $n$-dimensional vectors with nonnegative components.
\end{lem}

Yang and Yang \cite{YANGYANg2010} obtained the following conclusion.

\begin{lem}\label{lem4}
\emph{\cite{YANGYANg2010}} Let \( \mathcal{A} = \left( a_{i_1 i_2 \ldots i_r} \right) \)  be an $r$-order $n$-dimensional nonnegative tensor. Then 
\[    \underset{1\le i\le n}{\mathop{\min }}\,\sum\limits_{{{i}_{2}},\ldots ,{{i}_{r}}=1}^{n}{{{a}_{i{{i}_{2}}\ldots {{i}_{r}}}}}\le \mu \left( \mathcal{A} \right)\le \underset{1\le i\le n}{\mathop{\max }}\,\sum\limits_{{{i}_{2}},\ldots ,{{i}_{r}}=1}^{n}{{{a}_{i{{i}_{2}}\ldots {{i}_{r}}}}}.                   \]
\end{lem}

     Let $d_{G,r}(i)$ be the number of $r$-cliques containing vertex $i$ in $G$.  By Lemma \ref{lem4},
     we have 
     \begin{equation}
\underset{1\le i\le n}{\mathop{\min }}\,{{d}_{G,r}}\left( i \right)\le {{\mu }_{r}}\left( {{G}} \right)\le \underset{1\le i\le n}{\mathop{\max }}\,{{d}_{G,r}}\left( i \right).
\label{corss}
\end{equation}

For two $r$-order $n$-dimensional nonnegative tensors $\mathcal{A}$ and $\mathcal{B}$, if $\mathcal{A}-\mathcal{B}$ is a nonnegative but nonzero tensor, then we write $\mathcal{B} < \mathcal{A}$.

\begin{lem}\label{5}
\emph{\cite{Khan2015}}
Let $\mathcal{A}$ and $\mathcal{B}$ be two $r$-order $n$-dimensional nonnegative tensors. If $\mathcal{A} < \mathcal{B}$ and $\mathcal{B}$ is weakly irreducible, then $\mu(\mathcal{A}) < \mu(\mathcal{B})$.
\end{lem}

For an $r$-order $n$-dimensional nonnegative  tensor $\mathcal{A} = \left( a_{i_1 i_2 \ldots i_r} \right)$, the digraph of $\mathcal{A}$ is defined as the digraph $G_\mathcal{A}$ with vertex set $\{1,2,\cdots,n\}$ and the set of arcs $\{(i,j )| a_{ii_2\cdots i_r}\ne 0, j\in \{i_2,\cdots,i_r\}\}$. Furthermore, if $G_\mathcal{A}$ is strongly connected, then $\mathcal{A}$ is called weakly irreducible. Otherwise, the tensor is called weakly reducible.

\begin{lem}\label{lem2}
\emph{\cite{YangY2011}}  If $\mathcal{A}$ is nonnegative weakly irreducible tensor, then spectral radius $\mu \left( \mathcal{A} \right)$ is the unique positive eigenvalue of $\mathcal{A}$, with the unique positive eigenvector $x$, up to a positive scaling coefficient.
\end{lem}

\indent In 2023, Liu and Bu \cite{liu2023mantel} proposed the concept of \emph{$r$-clique connected}: an $r$-clique walk is a sequence of vertices and $r$-cliques \( v_1c_{r}^{(1)}v_2c_{r}^{(2)} \ldots v_m c_{r}^{(m)}v_{m+1} \) where \( c_{r}^{(i)} \) and \( c_{r}^{(i+1)} \) intersect in at least one vertex for \( i =1,2,\cdots,m-1 \). If any two vertices of a graph are connected to each other by $r$-clique walk, then the graph is called $r$-clique connected, which is equivalent to the following conclusion.
\begin{lem}\label{lem22}
  Let $G$ be a graph and $G'$  obtained from $G$ by deleting edges which are not contained in any $r$-cliques. Then  $G$ is $r$-clique connected if and only if $G'$  is connected.  
\end{lem}


\begin{lem}\label{lem3}
\emph{\cite{liu2023mantel}} For a graph $G$, $r$-clique tensor $\mathcal{A}_r(G)$  is weakly irreducible if and only if $G$ is $r$-clique connected.
\end{lem}

By Lemmas \ref{lem22} and \ref{lem3}, we obtain the following corollary. 
\begin{cor}\label{cor0}
    Let  $G$ be a graph and $G'$  obtained from $G$ by deleting edges which not contained in any $r$-cliques. Then $r$-clique tensor $\mathcal{A}_r(G)$  is weakly irreducible if and only if $G'$   is connected. 
\end{cor}

\begin{lem}\label{shao}\emph{\cite{shao2013}}
For a weakly reducible tensor $\mathcal{A}$, there is a lower triangular block tensor such that $\mathcal{A}$ is permutational similar to it and each diagonal block tensor is weakly irreducible. Furthermore, there is a diagonal block tensor such that its spectral radius is equal to $\mu(\mathcal{A})$.
\end{lem}

In fact, Since $G'$ obtained from $G$ by deleting edges which are not contained in any $r$-cliques, we have $\mathcal{A}_r\left( G \right)=\mathcal{A}_r\left( G' \right)$ and $\mu_r(G)=\mu_r(G').$ The following lemma can be directly derived from  Corollary \ref{cor0} and  Lemma \ref{shao}. 

\begin{lem}\label{dis}
    Let $G$ be a graph and let $G'$ obtained from $G$ by deleting edges which not contained in any $r$-cliques. If  $G$ is not  connected, then there exists 
    a connected component $H$ in $G'$ such that $\mu_r(H) = \mu_r(G')=\mu_r(G).$
\end{lem}

By Lemma \ref{5} and Corollary \ref{cor0}, we obtain the following corollary. 

\begin{cor}\label{55}
     Let $G$ be a connected graph and each edge in $G$ is contained in at least one $r$-clique. If $H$ is a proper subgraph of $G$, then $\mu_r(H)<\mu_r(G)$.
\end{cor}

Liu and Bu \cite{liu2023mantel} obtained the $r$-clique spectral radius of complete $r$-partite graphs.
\begin{lem}\label{r-bu}
\emph{\cite{liu2023mantel}} Let $G$ be a complete $r$-partite graph with vertex classes $V_1,V_2,\ldots,V_{r}$. Then
   $\mu_{r}(G)=\left( \prod_{i=1}^{r}|V_i| \right)^{\frac{r-1}{r}}.$
\end{lem}

Notice for $a\in \mathbb{R}$, \(\lfloor a \rfloor >a-1\). By Lemma \ref{r-bu}, for $1\le m\le r-1$, we can give the $r$-clique spectral radius of $K_m \vee T_{r-m}(n-m)$.
\begin{cor}\label{use} Let $ 2\le m+1 \le r \le n$. Then
   \begin{eqnarray}
    \mu_{r}(K_m \vee T_{r-m}(n-m))&= & \left( \prod_{i=1}^{r-m} 
 \left\lfloor \frac{n-m-1+i}{r-m}\right\rfloor \right)^{\frac{r-1}{r}}\nonumber 
     >\left(\frac{n-r}{r-m}\right)^{\frac{(r-1)(r-m)}{r}}. \nonumber
\end{eqnarray}
\end{cor}

\begin{lem}\label{lem6}\emph{\cite{yuan20221}}
Let $r\ge 3$ and $n\ge2kr^3$. Then 
\[ex(n,K_r,(k+1)K_r)\le kr^2\left(\frac{n-kr}{r-1}\right)^{r-1}.  \]
\end{lem}
\begin{lem}\label{use2}
   Let $n\ge 5$, $G=K_3\vee (n-3)K_1$. Then \[\mu_3(G)=\left( \sqrt[3]{\sqrt{3}(n-3) + \sqrt{3(n-3)^2 - \frac{1}{27}}} + \sqrt[3]{\sqrt{3}(n-3) - \sqrt{3(n-3)^2 - \frac{1}{27}}} \right)^2.\]
\end{lem}

  \begin{proof}
 By Lemmas \ref{lem2} and  \ref{lem3}, there exists a unique  positive eigenvector $x=(x_1,\cdots,x_n)^{\top}$  corresponding to \( \mu_3(G) \) with $x_1^3+\cdots+x_n^3=1$. According to symmetry, the vertex set could be partitioned into two parts ($K_3$ and $(n-3)K_1$), and the eigenvector components in each part are equal, we have
\begin{equation}
\begin{cases}
 &\mu_{3}(G) x_{1}^{2}= x_1^2+2(n-3)x_{1}x_{2}   \nonumber\\ 
 &\mu_{3}(G) x_{2}^{2}=3x_{1}^{2}    \nonumber   
\end{cases}
\end{equation}
From the above equations, we have 
\[
\mu_3(G) = \left( \sqrt[3]{\sqrt{3}(n-3) + \sqrt{3(n-3)^2 - \frac{1}{27}}} + \sqrt[3]{\sqrt{3}(n-3) - \sqrt{3(n-3)^2 - \frac{1}{27}}} \right)^2.
\]
\end{proof}

\section{Proofs of Theorems \ref{sun2} and \ref{cor1} }

\begin{proof}[\textbf{Proof of Theorem \ref{sun2}}] 
For $n\ge (2r-k-1)\left(\mathrm{e}r^4k^{\frac{k}{2}}\right)^{\frac{k}{2k-2r+1}}+k $ and $3\le r\le k\le 2r-1$, let $G$ be a $2K_r$-free graph on $n$ vertices with maximum $\mu_k(G)+\mu_{k+1}(G)+\cdots+\mu_{2r-1}(G)$. We delete all the edges in $G$ which are not contained in any $k$-cliques, forming $G'$, and $\mu_l(G')=\mu_l(G)$ for $l=k,k+1,\cdots,2r-1$. Then $G'$ is a $2K_r$-free graph on $n$ vertices with maximum $\mu_k(G')+\mu_{k+1}(G')+\cdots+\mu_{2r-1}(G')$, and each edge in $G'$  is contained in at least one $k$-clique.

The sketch of proof  is as follows: 
For $k=2r-1$, $G'\cong K_{2r-1}\cup(n-2r+1)K_1$ is obvious. For $r\le k\le 2r-2$,  we first
 show all the $k$-cliques in $G'$ share exactly $2k-2r+1$ common vertices (see \textbf{Claim \ref{claim1}}).  Then we prove  $G'$ is a connected graph (see \textbf{Claim \ref{claim2.1}}). Furthermore, we prove $G'$ is a complete $k$-partite graph (see \textbf{Claim \ref{claim2} Step 1}). Finally, we prove
 $G' \cong K_{2k-2r+1} \vee T_{2r-k-1}(n-2k+2r-1)$ (see \textbf{Claim \ref{claim2} Step 2}). Next we begin the proof.

    If $k=2r-1$, we have the number of $(2r-1)$-cliques in $G'$ is at most one. Otherwise, $G'$ has a copy of $2K_r$. Since each edge in $G'$  is contained in at least one $(2r-1)$-clique, we can easily obtain $G'\cong K_{2r-1}\cup(n-2r+1)K_1$. 

Next we consider that $r\le k\le 2r-2$. We claim that the number of $k$-cliques in $G'$ is at least two. If $G'$ has only one $k$-clique, then $\sum\limits_{i=k}^{2r-1}\mu_i(G')=\mu_k(K_k)$. Since $K_{k+1}$ is $2K_r$-free, by Corollary \ref{55}, we have 
 $\sum\nolimits_{i=k}^{2r-1}\mu_i(K_{k+1})>\mu_k(K_{k+1})>\mu_k(K_{k})
= \sum\nolimits_{i=k}^{2r-1}\mu_i(G')$, a contradiction.

\begin{lem}\label{free}
    For $r\le k\le 2r-2$, let $H$ be an $n$-vertices graph with more than two $k$-cliques and  each edge in $H$ is contained in some $k$-clique. Then $H$ is $2K_r$-free if and only if
 any two $k$-cliques in $H$ share at least $2k-2r+1$ vertices.   
\end{lem}
\begin{proof}
    If $H$ is $2K_r$-free, then any two $k$-cliques in $H$ share at least $2k-2r+1$ vertices. Otherwise their union would contain at least $2r$ vertices and form a copy of $2K_r$, a contradiction. 
    
    If any two $k$-cliques in $H$  share at least $2k-2r+1$ vertices, the union of two $k$-cliques in $H$  has at most $2r-1$ vertices.
    Recall that each edge in $H$  is contained in at least one $k$-clique, any $r$-cliques only be found in $k$-cliques, which means $H$ is $2K_r$-free.
\end{proof}
By Lemma \ref{free}, we have any two $k$-cliques in $G'$ must share at least $2k-2r+1$ vertices. Similarly, for $0\le j\le 2r-k-1$, any $k$-clique and $(k+j)$-clique in $G'$ share at least $2k-2r+1+j$ vertices (\textbf{this fact will be used in Claim \ref{claim1}}). Otherwise their union would contain at least $2r$ vertices and form a copy of $2K_r$, a contradiction. 


\begin{claim}\label{claim1}
All the $k$-cliques in $G'$ share exactly $2k-2r+1$ common vertices.
\end{claim}

\begin{proof}
Let $F_{r+1,2k-2r+1}$ be a graph with $(r+1)(2r-1-k)+2k-2r+1$ vertices  consisted of $r+1$ cliques with order $k$ (denote these \( k \)-cliques as $K^{(1)}, K^{(2)},\cdots, K^{(r+1)}$),  which intersect in exactly $2k-2r+1$ common vertices (denote as $\{v_1,\cdots,v_{2k-2r+1}\}$). 
 We call $\mathcal{K}=\{v_1,\cdots,v_{2k-2r+1}\}$ the kernel and $K^{(1)}\backslash \mathcal{K} ,\cdots,  K^{(r+1)}\backslash \mathcal{K}$ the  petals of $F_{r+1,2k-2r+1}$.
 
 If $G'$ has a copy of $F_{r+1,2k-2r+1}$ with the kernel $\mathcal{K}$,  we will use proof by contradiction to show all the $k$-cliques in $G'$ contain $\mathcal{K}$. Suppose  $K^{(0)}$ is a $k$-clique in $G'$ such that $|\mathcal{K} \cap K^{(0)}|=l$, where $0\le l \le 2k-2r$. Any two $k$-cliques in $G'$ share at least $2k-2r+1$ vertices implies that $ K^{(0)} $ intersect at most
$\left\lfloor \frac{k - l}{2k - 2r + 1 - l} \right\rfloor$
petals of the subgraph \( F_{r+1,2k-2r+1} \).  We can easily get $\left\lfloor \frac{k - l}{2k - 2r + 1 - l} \right\rfloor$ is monotonically increasing with respect to $l$, then we have $ \left\lfloor \frac{k - l}{2k - 2r + 1 - l} \right\rfloor\le 2r-k\le r$. Since $r<r+1$, we can find a $k$-clique (without loss of generality let $K^{(1)}$) in \( F_{r+1,2k-2r+1} \) such that  $|K^{(1)}\cap K^{(0)}|<2k-2r+1$, which contradicts any two $k$-cliques in $G'$ share at least $2k-2r+1$ vertices. Thus all the $k$-cliques in $G'$ share  $\mathcal{K}$.  Since the $k$-cliques in $F_{r+1,2k-2r+1}$ share exactly kernel $\mathcal{K}$,
thus all the $k$-cliques in $G'$ share exactly $2k - 2r + 1$ common vertices.

Next we consider the case that  $G'$ does not have a copy of $F_{r+1,2k-2r+1}$. Recall that $N_{G}(i_1,\cdots,i_{s})= N_{G}(i_1)\cap N_{G}(i_2)\cap \cdots \cap N_{G}(i_s) $. Let $K$ be a $k$-clique in $G'$. 
Since  every $(k+j)$-clique of $G'$ intersects $K$ in at least $2k-2r+1+j$ vertices for $0\le j\le 2r-k-1$,  we have 
\[ |C_{k+j}(G')|\le \sum\limits_{\{i_1,\cdots,i_{2k-2r+1+j}\}\subset  V(K)}|C_{2r-k-1}(G'[N_{G'}(i_1,\cdots,i_{2k-2r+1+j})])|.  \]

Since $G'$ does not have a copy of $F_{r+1,2k-2r+1}$, we have $G'[N_{G'}(i_1,\cdots,i_{2k-2r+1+j})]$ is $(r+1)K_{2r-k-1}$-free for $\{i_1,\cdots,i_{2k-2r+1+j}\}\subset V(K)$. By Lemma \ref{lem6}, we have 
\begin{align}
&|C_{2r-k-1}(G'[N_{G'}(i_1,\dots,i_{2k-2r+1+j})])| \nonumber \\
&\le \operatorname{ex}(n-2k+2r-1-j,K_{2r-k-1},(r+1)K_{2r-k-1}) \nonumber \\
&\le r(2r-k-1)^2\left(\frac{n-2k+2r-1-j-r(2r-k-1)}{2r-k-2}\right)^{\!2r-k-2}.\nonumber
\end{align}
Thus,  we have 
 \begin{eqnarray}
    |C_{k+j}(G')|&\le & \tbinom{k}{2k-2r+1+j} r(2r-k-1)^2\left(\frac{n-2k+2r-1-j-r(2r-k-1)}{2r-k-2}\right)^{2r-k-2}\nonumber   \\
     & < & \tbinom{k}{2k-2r+1+j}r^3\left(\frac{n-k}{2r-k-2}\right)^{2r-k-2}\nonumber \\
     &<& k^{\frac{k}{2}}r^3\left(\frac{2r-k-1}{2r-k-2}\right)^{2r-k-2}\left(\frac{n-k}{2r-k-1}\right)^{2r-k-2}\nonumber\\
     &\le&\mathrm{e}k^{\frac{k}{2}}r^3\left(\frac{n-k}{2r-k-1}\right)^{2r-k-2}\nonumber, 
\end{eqnarray}
The last inequality is obtained by
$$\left(\frac{2r-k-1}{2r-k-2}\right)^{2r-k-2}=\left(1+\frac{1}{2r-k-2}\right)^{2r-k-2}\le \mathrm{e}.$$

\noindent For $n\ge (2r-k-1)\left(\mathrm{e}r^4k^{\frac{k}{2}}\right)^{\frac{k}{2k-2r+1}}+k $, by Equation \eqref{corss} and Corollary \ref{use}, we have
 \begin{eqnarray}
\sum\limits_{i=k}^{2r-1}\mu_i(G')        &\le& \sum\limits_{i=k}^{2r-1} \underset{1\le j\le n}{\mathop{\max }}\,{{d}_{G',i}}\left( j \right) \le
\sum\limits_{i=k}^{2r-1}|C_i(G')| \nonumber   \\
&<& (2r-k)\mathrm{e}k^{\frac{k}{2}}r^3 \left(\frac{n-k}{2r-k-1}\right)^{2r-k-2} \nonumber 
\le \mathrm{e}r^4k^{\frac{k}{2}}\left(\frac{n-k}{2r-k-1}\right)^{2r-k-2} \\
&= & \mathrm{e}r^4k^{\frac{k}{2}}     \left (\frac{n-k}{2r-k-1}\right)^{\frac{(k-1)(2r-k-1)}{k}}   \left (\frac{2r-k-1}{n-k}\right)^{\frac{2k-2r+1}{k}} \nonumber \\
&\le& \mathrm{e}r^4k^{\frac{k}{2}}   \left   (\frac{n-k}{2r-k-1}\right)^{\frac{(k-1)(2r-k-1)}{k}}   \left (\frac{1}{\mathrm{e}r^4k^{\frac{k}{2}}}\right)    
=    \left  (\frac{n-k}{2r-k-1}\right)^{\frac{(k-1)(2r-k-1)}{k}}             \nonumber \\  &<&\mu_k(K_{2k-2r+1} \vee T_{2r-k-1}(n-2k+2r-1)),\nonumber
\end{eqnarray}
which contradicts $G'$ is a $2K_r$-free graph on $n$ vertices with maximum $\sum\nolimits_{i=k}^{2r-1}\mu_i(G')$. Thus, all the $k$-cliques in $G'$ share exactly $2k-2r+1$ common vertices.
\end{proof}

By Claim \ref{claim1}, we have all the $k$-cliques in $G'$ share exactly common $2k-2r+1$ vertices. We claim that $G'$ is $K_{k+1}$-free (\textbf{this fact will be used in Claim \ref{claim2.1}}). Otherwise, let $K_{k+1}$ be a subgraph of $G'$ and $V(K_{k+1})=\{v_1,\cdots,v_{k+1}\}$. Without loss of generality,  all the $k$-cliques in $G'$ share exactly $\{v_1,\cdots,v_{2k-2r+1}\}$, then there exists a $k$-clique $\{v_2, v_3,\cdots, v_{k+1}\}$ in $G'$ which does not share  $\{v_1,\cdots,v_{2k-2r+1}\}$, a contradiction. 

Since $G'$ is $K_{k+1}$-free , we have $\sum\nolimits_{i=k}^{2r-1}\mu_i(G')=\mu_k(G')$, it follows that the problem of the maximum value of the sum from $k$-clique spectral radius to $(2r-1)$-clique spectral radius is transformed into the problem of the maximum $k$-clique spectral radius. 

\begin{claim}\label{claim2.1}
  $G'$ is a connected graph.
\end{claim} 
\begin{proof}
 Suppose that $G'$ is not connected. Without loss of  generality, let $G'$ has $m$ connected components $G_1',\cdots,G_m'$ and $\mu_k(G_1')=\mu_k(G')>0$ by Lemma \ref{dis}. If there exist a $k$-clique in  $G_i'$ for $2\le i\le m$, then we can find two disjoint  $r$-cliques from $G_1'$ and $G_i'$ respectively, which contradicts $G'$ is $2K_r$-free. Since each edge in $G'$  is contained in at least one $k$-clique, we have $G'$ is a union of a connected component $G_1'$ with $\mu_k(G_1')=\mu_k(G')>0$ and some isolated vertices. 

By Claim \ref{claim1}, let all the $k$-cliques in $G'$ share $\{u_1,\allowbreak\dots,\allowbreak u_{2k-2r+1}\}$ and
$ \{u_1,\allowbreak\dots,\allowbreak u_{2k-2r+1},\cdots,u_k\}$ be a $k$-clique of $G'$. For any isolated vertex $u_0$  of $G'$, we can add the edges between $u_0$ and $\{u_1,\cdots,u_{k-1}\}$ to obtain a new $k$-clique $\{u_0,u_1,\cdots,u_{k-1}\}$, and the new graph is denoted as $G_1''$. All
the $k$-cliques in $G_1''$ share $\{u_1,\allowbreak\dots,\allowbreak u_{2k-2r+1}\}$, by Lemma \ref{free} , $G_1''$ is $2K_r$-free. And by Corollary \ref{55}, we have $\mu_k(G_1'') > \mu_k(G_1') = \mu_k(G')$, a contradiction. Thus $G'$ is a connected graph.
\end{proof}

\begin{claim}\label{claim2}
   $G' \cong K_{2k-2r+1} \vee T_{2r-k-1}(n-2k+2r-1).$
\end{claim} 

\begin{proof}
  By  Claim \ref{claim2.1} and Corollary \ref{cor0}, there exists a unique  positive eigenvector $x=(x_1,\cdots,x_n)^{\top}$  corresponding to \( \mu_k(G) \) with $x_1^k+\cdots+x_n^k=1$.
We will prove Claim \ref{claim2} in two steps. First, we prove $G'$ is a complete $k$-partite graph (see \textbf{Step 1} below). Second, we prove $G' \cong K_{2k-2r+1} \vee T_{2r-k-1}(n-2k+2r-1)$ (see \textbf{Step 2} below).

\textbf{Step 1: }$G'$ is a complete $k$-partite graph.

Suppose the contrary that  $G'$ is not a complete $k$-partite graph, which implies that it is a complete $t$-partite graph for $ t \ne k$ or there are non-adjacent vertices $u, v$ in $V (G') $ satisfying $N_{G'}(u)\ne N_{G'}(v)$.

\textbf{Case 1.1: }Suppose $G'$ is a complete $t$-partite graph, where $t\ne k$. If $k>t$, then we have $\mu_k(G') =0$. Obviously a contradiction. If $t > k$, then $G'$ has a copy of $K_{k+1}$, this contradicts that $G'$ is $K_{k+1}$-free.

\textbf{Case 1.2: }Suppose that  there exist two non-adjacent vertices $u$ and $v$ with $N_{G'}(u)\ne N_{G'}(v)$. Without loss of generality, let $w\sim u$ and $w\not \sim v$.

For any vertex $i\in[n]$, we define 
 \[ W_{G'}(i,x)     =  \sum\limits_{\{{i},{i_2},\cdots,{i_k}\}\in C_{k}\left( G' \right)}{x_{i_2}\cdots x_{i_k}    }.\]

\textbf{Case 1.2.1: }  $W_{G'}(v,x)<W_{G'}(u,x)$ or $W_{G'}(v,x)<W_{G'}(w,x)$.

 If $W_{G'}(v,x)<W_{G'}(u,x)$, we delete all the edges between $v$ and $N_{G'}(v)$, after that, connecting $v$ to all the vertices in $N_{G'}(u)$ and we get a $K_{k+1}$-free graph $H_1$. Now we prove that $H_1$ is $2K_r$-free.
  
 By Claim \ref{claim1}, without loss of generality, let all the $k$-cliques in $G'$ share $\{u_1,\allowbreak\dots,\allowbreak u_{2k-2r+1}\}$.  Recall that each edge in $G'$  is contained in at least one $k$-clique. Since $G'$ is a connected graph, we have $N_{G'}(u_1,\cdots,u_{2k-2r+1})=V(G')\backslash \{u_1,\cdots,u_{2k-2r+1}\}$ and 
 \[N_{H_1}(u_1,\cdots,u_{2k-2r+1})=V(H_1)\backslash \{u_1,\cdots,u_{2k-2r+1}\}.\]
 We claim that each $k$-clique in $H_1$ contains $\{u_1,\cdots,u_{2k-2r+1}\}$. Otherwise, suppose there exists a $k$-clique $K$  in $H_1$ such that $u'\notin K$ where $u'\in \{u_1,\cdots,u_{2k-2r+1}\}$. By $N_{H_1}(u_1,\cdots,u_{2k-2r+1})=V(H_1)\backslash \{u_1,\cdots,u_{2k-2r+1}\}$, we have $H_1[\{u'\}\cup K]$  is a copy of $K_{k+1}$, which  contradicts $G'$ is $K_{k+1}$-free. Thus $H_1$ is $2K_r$-free.
 
  By Lemma \ref{lem1}, we have 
\begin{eqnarray}
    \mu_k(H)& \ge & {{x}^{\top}} \mathcal{A}_k\left( {{H_1}} \right){{x}^{k-1}}  \nonumber   \\
     & = & {{x}^{\top}}\mathcal{A}_k\left( {G'} \right){{x}^{k-1}} -kx_vW_{G'}(v,x) +kx_vW_{G'}(u,x) \nonumber  \\
     & > & {{x}^{\top}} \mathcal{A}_k\left( {G'} \right){{x}^{k-1}} = \mu_k(G') . \nonumber
\end{eqnarray}
This contradicts the choice of $G'$. If $W_{G'}(v,x)<W_{G'}(w,x)$, the proof is similar. 

\textbf{Case 1.2.2: }  $W_{G'}(v,x)\ge W_{G'}(u,x)$ and $W_{G'}(v,x)\ge W_{G'}(w,x)$. 

We delete all the edges between $u$ and $N_{G'} (u)$, after that, connecting $u$ to all the vertices in $N_{G'} (v)$. And carry out the same operation on $w$. Then we get a $K_{r+1}$-free graph $H_2$.  Similar to the proof of Case 1.2.1, we have  $H_2$ is $2K_r$-free.  And by Lemma \ref{lem1}, we have 
\begin{eqnarray}
    {{\mu }_{k}}( {{H_2}})& \ge & {{x}^{\top}} \mathcal{A}_k\left( {{H_2}} \right){{x}^{k-1}}  \nonumber   \\
     & = & {{x}^{\top}}\mathcal{A}_k\left( {G'} \right){{x}^{k-1}} -kx_uW_{G'}(u,x)+ kx_uW_{G'}(v,x) \nonumber   \\
     & - &kx_wW_{G'}(w,x) +kx_wW_{G'}(v,x)+ k\sum\limits_{\{u,w,{i_3},\cdots,{i_k}\}\in C_{k}\left( G' \right)}x_ux_wx_{i_3}\cdots x_{i_k}. \nonumber    
\end{eqnarray}

Since $x$ is positive and each edge in $G'$  is contained in at least one $k$-clique, we have 
\[ \sum\limits_{\{u,w,{i_3},\cdots,{i_k}\}\in C_{k}\left( G' \right)}x_ux_wx_{i_3}\cdots x_{i_k}>0. \]
Thus, we have ${\mu }_{k}( {{H_2}})> \mu_k(G')$. It is also contradictory.

Through the discussions of the cases above, we can always get a  $2K_r$-free graph with larger $k$-clique spectral radius than $G'$, a contradiction.  Therefore, $G'$ is
 a complete  $k$-partite graph.
Next, we will prove $G' \cong K_{2k-2r+1} \vee T_{2r-k-1}(n-2k+2r-1).$

\textbf{Step 2: } $G' \cong K_{2k-2r+1} \vee T_{2r-k-1}(n-2k+2r-1).$

   By Step 1, we have $G'$ is a complete  $k$-partite graph with vertex classes $V_1,\ldots,V_{k}$. For $i\in[k]$, let $n_k=|V_k|$. Since all the $k$-cliques in $G'$ share exactly $2k-2r+1$ common vertices, then there is exactly $2k-2r+1$ partite sets with exactly one vertex. Without loss of  generality, let $ n_{2r-k}=\cdots=n_{k}=1$.
   
 By Lemma \ref{r-bu}, we have  $\mu_k(G')= \left(n_{1}n_{2}\cdots n_{2r-k-1}\right)^{\frac{k-1}{k}}$. Thus, for $i,j\in \{ 1,\cdots,2r-k-1\}$, when $|n_i-n_j|\le 1$, $\mu_k(G')$ is maximum. Thus,  $G' \cong K_{2k-2r+1} \vee T_{2r-k-1}(n-2k+2r-1).$ 
\end{proof}

In conclusion, For $3\le r\le k\le 2r-2$ and $n\ge (2r-k-1)\left(\mathrm{e}r^4k^{\frac{k}{2}}\right)^{\frac{k}{2k-2k+1}}+k $, we have $G'\cong  K_{2k-2r+1} \vee T_{2r-k-1}(n-2k+2r-1)$, and any new edge can not be added in $G'$, otherwise forming new $k$-cliques, a contradiction. Thus we have $G\cong  K_{2k-2r+1} \vee T_{2r-k-1}(n-2k+2r-1)$.  For $ k= 2r-1$, we have $G'\cong K_{2r-1}\cup(n-2r+1)K_1$, then  $|C_k(G)|=1$.
\end{proof}

Next, we will prove Theorem \ref{cor1}.

\begin{proof}[\textbf{Proof of Theorem \ref{cor1}}]
Let $G$ be a $2K_3$-free graph on $n$ vertices with maximal $3$-clique spectral radius. For $n=3,4,5$, it is easy to get $G\cong K_n$. Next we consider $n\ge 6$. 
We delete all the edges in $G$ which are not contained in any $3$-cliques, forming $G'$, and $\mu_3(G')=\mu_3(G)$. Then $G'$ is a $2K_3$-free graph on $n$ vertices with maximum $3$-clique spectral radius, and each edge in $G'$  is contained in at least one $3$-clique.

Since $G'$ is $2K_3$-free, we have $G'$ is $K_6$-free. Thus, we will consider the following three cases:

  \textbf{Case 1: }If $G'$ has a copy of $K_5$.
  
   Since each edge in $G'$  is contained in at least one $3$-clique, we have $G'\cong K_5\cup(n-5)K_1$. Otherwise, $G'$ has a copy of $2K_3$. 
   
  \textbf{Case 2: }If $G'$ is $K_5$-free and $G'$ has a copy of $K_4$.
  
  Let  $L=\{u,v,u_1,v_1\}$ be a 4-clique in $G'$ and $G''=G'[V(G')\backslash L]$. It is easy to obtain $G'' $ is $K_3$-free.  We claim that $G''$ is a union of some isolated vertices. Otherwise, let $\{a,b\}$ be an edge in $G''$, since  each edge in $G'$  is contained in at least one $3$-clique, we have a vertex (with loss of generality let $u$) in $L$ and $\{a,b\}$ form a 3-clique in $G'$. Then $\{a,b,u\}$ and $\{v,u_1,v_1\}$ are two disjoint 3-clique in $G'$, which contradicts $G'$ is $2K_3$-free.
  
  We claim that any vertex $w\in V(G'')$ is adjacent to at most three vertices of $L$ in $G'$. Otherwise, if $w$ is adjacent to all vertices in  $L$, then $\{w,u,v,u_1,v_1\}$ is a 5-clique in $G'$, which contradicts $G'$ is $K_5$-free.
  
  If there exists a vertex $w\in V(G'')$ such that $w$ is adjacent to three vertices in $L$, denoted as $\{u,v,u_1\}$, then we claim any vertices   in $V(G'')$ can not be adjacent to $v_1$.  Otherwise, suppose that there exists a
  vertex $w_1$ in $V(G'')$ such that $\{w_1,v_1\}$ is a edge of $G'$. Since $G'[V(G')\backslash \{u,v,u_1\}]$ is $K_3$-free and each edge in $G'$ is contained in at least one 3-clique, there exists a vertex in $\{u,v,u_1\}$ (without loss of generality let $u$) such that $\{w_1,v_1,u\}$ is a 3-clique. Then $\{w,v,u_1\}$ and $\{w_1,v_1,u\}$ are two disjoint 3-cliques in $G'$, which  contradicts $G'$ is $2K_3$-free. Thus, we have $G'\subseteq K_3\vee (n-3)K_1$ and $\mu_3(G')\le \mu_3(K_3\vee (n-3)K_1)$ by Corollary \ref{55}. Then $G'\cong K_3\vee (n-3)K_1$.

  If there does not exist a vertex $w\in V(G'')$ such that $w$ is adjacent to three vertices in $L$, then any $w_1\in V(G'')$ adjacent to at most two vertices in $L$. Thus, we have $|C_3(G')|\le 4+(n-4)=n$. Since $r=3$ and $\omega=4$, by Theorem \ref{Liu} we have
 \begin{eqnarray}
\mu_3(G')\le  \frac{r}{\omega}\tbinom{\omega}{r}^{\frac{1}{r}}|C_r(G)|^{\frac{r-1}{r}}= \frac{3}{4} \tbinom{4}{3}^{\frac{1}{3}}|C_3(G')|^{\frac{2}{3}}
\le \frac{3}{4}\tbinom{4}{3}^{\frac{1}{3}}n^{\frac{2}{3}}=\frac{3}{4}(2n)^{\frac{2}{3}}
.\nonumber
\end{eqnarray}

We will show $\mu_3(G')\le \frac{3}{4}(2n)^{\frac{2}{3}}< \mu_3(K_3\vee (n-3)K_1)$.
Let $g(n)=\mu_3(K_3\vee (n-3)K_1)^{\frac{3}{2}}-2(\frac{3}{4})^{\frac{3}{2}}n.$ By Lemma \ref{use2}, we have $\mu_3(K_3\vee (n-3)K_1)^{\frac{3}{2}}> 2\sqrt{3}(n-3)$. Then $g(n)>2\sqrt{3}(n-3)-2(\frac{3}{4})^{\frac{3}{2}}n.$ We can easily get $g(n)$ is monotonically increasing for $n>6$ and $g(6)>0$.
Since for $n\ge 6$, we have $\mu_3(G')<\frac{3}{4}(2n)^{\frac{2}{3}}< \mu_3(K_3 \vee (n-3)K_1)$, which contradicts $G'$ has maximum 3-clique spectral radius.

  \textbf{Case 3: } If $G'$ is $K_4$-free.
  
Since $G'$ is $2K_3$-free, for any $ K \in C_3(G')$, $G'[V(G') \backslash K]$ is $K_3$-free. Without loss of generality, let $\{u,v,w\}$ be a 3-clique in $G'$, $G''=G'[V(G') \backslash \{u,v,w\}]$. 

 We claim that any  $w_1\in V(G'')$ is adjacent to at most two vertices in $\{u,v,w\}$. Otherwise, if $w_1$ is adjacent to all vertices in $\{u,v,w\}$, then $\{u,v,w,w_1\}$ is a 4-clique in $G'$, which contradicts $G'$ is $K_4$-free. 
 
 If $|E(G'')|=0$, we have $G''$ is a union of some isolated vertices. Thus, we have $G'$ is a proper subgraph of $ K_3\vee (n-3)K_1$ and $\mu_3(G')<\mu_3(K_3\vee (n-3)K_1)$ by Corollary \ref{55}, a contradiction.

 Then $|E(G'')|>0$. Let $k_3(e)$ be the number of $K_3$ in $G'$ containing $e\in E(G'')$. We claim that $k_3(e)=1$ for any $e\in E(G')$. Otherwise, let $\{a,b\}$ be an edge in $G''$
 with $k_3(\{a,b\})\ge 2$. Without loss of generality, let  $u$ and $v$ be the two vertices in $\{u,v,w\}$ such that $\{a,b,u\}$ and $\{a,b,v\}$ are two 3-cliques in $G'$, then $\{a,b,u,v\}$ is a 4-clique in $G'$, which contradicts $G'$ is $K_4$-free.

Thus, we have each edge in $G''$ is contained in exactly one $3$-clique with the third vertex being one of $u,v,w$. Let
\begin{eqnarray}
&E_u& = \{ab\in E(G''): \{a,b,u\} \text{~~is a $3$-clique in~~} G'  \}, \nonumber \\
&E_v& = \{ab\in E(G''): \{a,b,v\} \text{~~is a $3$-clique in~~} G'  \}, \nonumber \\
&E_w& = \{ab\in E(G''): \{a,b,w\} \text{~~is a $3$-clique in~~} G'  \}. \nonumber
\end{eqnarray}
and  let 
\begin{eqnarray}
&V_{uv}& = \{y\in V(G''): \{u,v,y\} \text{~~is a $3$-clique in~~} G' \},  \nonumber \\
&V_{uw}& = \{y\in V(G''): \{u,w,y\} \text{~~is a $3$-clique in~~} G' \},  \nonumber \\
&V_{vw}& = \{y\in V(G''): \{v,w,y\} \text{~~is a $3$-clique in~~} G' \}. \nonumber 
\end{eqnarray}
Then $E(G'')=E_u\cup E_v\cup E_w$. For each $e\in E_u$, we say that $e$ has color $u$. Similarly for edges in $E_v$ and $E_w$.

For $i=1,2$, define $S_i=\{T: T \text{~~is a $3$-clique in~~}  G' \text{~~and~~} |V(T)\cap \{u,v,w\}|=i \}$. Then we have 
\[|C_3(G')|= |S_1|+|S_2|+1.\]
 
 \textbf{Case 3.1: } There exist two edges of different colors.

Assume that $\{a,b\}\in E_u$ and $\{c,d\}\in E_v$. If $\{a,b\}\cap \{c,d\}=\phi$, then $\{a,b,u\}$ and $\{c,d,v\}$ are two disjoint 3-clique in $G'$, which contradicts $G'$ is $2K_3$-free. Thus, $\{a,b\}\cap \{c,d\}=1$. Without
loss of generality, we suppose that  $b=d$. Since $G''$ is $K_3$-free, $\{a,c\}$ is not an edge. For each $\{y_1,y_2\}\in E(G'')$, if $\{y_1,y_2\}\cap\{a,b,c\}=\phi$, then the 3-clique containing $\{y_1,y_2\}$ is disjoint from one 3-clique of $\{a,b,u\}$ and $\{c,d,v\}$, which contradicts $G'$ is $2K_3$-free. Thus, each edge in $G''$ has a vertex in $\{a,b,c\}$. Let $A=N_{G''}(a)\backslash \{b\}$,  $B=N_{G''}(b)\backslash \{a,c\}$ and $C=N_{G''}(c)\backslash \{b\}$. Since $G''$ is $K_3$-free, we have $A\cap B=\phi$ and $B\cap C=\phi$. Thus 
\[ |A|+|B|\le n-6 \text{~~and~~} |B|+|C|\le n-6.   \]
Then, 
\[ |S_1|= |A|+|B|+|C|+2\le 2(n-6)+2-|B|\le 2n-10. \]

Since $|V(G'')|= n-3$ and any $w_1\in V(G'')$ is adjacent to at most two vertices in $\{u,v,w\}$, we have
\[|C_3(G')|= |S_1|+|S_2|+1\le 2n-10+n-3+1=3n-12. \]
Since $r=3$ and $\omega=3$, by Theorem \ref{Liu}, we have
 \begin{eqnarray}
\mu_3(G') \le  \frac{r}{\omega}\tbinom{\omega}{r}^{\frac{1}{r}}|C_r(G')|^{\frac{r-1}{r}} = |C_3(G')|^{\frac{2}{3}}
\le (3n-12)^{\frac{2}{3}}
.\nonumber
\end{eqnarray}

We will show $\mu_3(G') \le (3n-12)^{\frac{2}{3}}< \mu_3(K_1\vee T_2(n-1))$.
Let $f(n)=(\frac{n-3}{2})^2-(3n-12)$, which  is monotonically increasing for $n> 14$ and $f(14)>0$. Then for $n\ge 14$, we have $\mu_3(G') \le (3n-12)^{\frac{2}{3}} < (\frac{n-3}{2})^{\frac{4}{3}} < ( \lfloor \frac{n-1}{2} \rfloor  \lfloor \frac{n}{2} \rfloor )^{\frac{2}{3}}=\mu_3(K_1\vee T_2(n-1))$. For $6\le n\le 13$, by  direct calculation we can obtain $\mu_3(G') \le (3n-12)^{\frac{2}{3}}  < ( \lfloor \frac{n-1}{2} \rfloor  \lfloor \frac{n}{2} \rfloor )^{\frac{2}{3}}=\mu_3(K_1\vee T_2(n-1))$, which contradicts $G'$ has maximum 3-clique spectral radius.

  \textbf{Case 3.2: } Each edge in $G''$ has the same color.
 
Without loss of generality, let $E(G'')=E_u$ and $\{a,b\}$ be an edge in $G''$. We claim $V_{vw}=\emptyset$. Otherwise, let $x\in V_{vw}$. If $x\notin \{a,b\}$, then  $\{u,a,b\}$ and $\{x,v,w\}$ are two disjoint 3-cliques in $G'$, which contradicts $G'$ is $2K_3$-free. If $x=a( x=b)$, then  $\{u,v,w,a\}(\{u,v,w,b\})$ is a 4-clique in $G'$, which contradicts $G'$ is $K_4$-free.
Thus, $V_{vw}=\emptyset$, which means all the 3-cliques in $G'$  share exactly one vertex $u$. 

Similar to the proof of  Claim \ref{claim2.1}, we can obtain  $G'$ is connected. Since all the 3-cliques in $G'$  share exactly one vertex $u$ and $G'$ is connected, similar to the proof of  Claim \ref{claim2}, we can obtain $G'\cong   K_{1}\vee T_{2}(n-1)$.

Through our discussion on Cases 1, 2 and 3, we have identified all possible extremal graphs that attain the maximum 3-clique spectral radius. Next, we compare the 3-clique spectral radius of these extremal graphs to determine which is largest.

Let $G_1= K_1\vee T_2(n-1)$, $G_2= K_3\vee (n-3)K_1 $ and  $G_3= K_5\cup (n-5)K_1$. $\mu_3(G_1)$ and  $\mu_3(G_2)$
have been given in Corollary \ref{use} and Lemma \ref{use2}.
We can easily  obtain $\mu_3(G_3)=\mu_3(K_5)=6$.

We claim that

$(i)$ For $n=6$, $\mu_3(G_3)> \mu_3(G_1)$ and $\mu_3(G_3)> \mu_3(G_2)$.

$(ii)$ For $7\le n\le 13$, $\mu_3(G_2)> \mu_3(G_1)$ and $\mu_3(G_2)> \mu_3(G_3)$.

$(iii)$ For $n\ge 14$, $\mu_3(G_1)> \mu_3(G_2)$ and $\mu_3(G_1)> \mu_3(G_3)$.  

By Lemma \ref{r-bu}, $\mu_3(G_1)=  ( \lfloor \frac{n-1}{2} \rfloor  \lfloor \frac{n}{2} \rfloor )^{\frac{2}{3}}$. Since  $r=3$, $|C_3(G_2)|=3n-8$ and $\omega(G_2)=4$  , by Theorem \ref{Liu}, we 
   have 
\begin{eqnarray}
\mu_3(G_2)\le  \frac{r}{\omega}\tbinom{\omega}{r}^{\frac{1}{r}}|C_r(G_2)|^{\frac{r-1}{r}}= \frac{3}{4} \tbinom{4}{3}^{\frac{1}{3}}(3n-8)^{\frac{2}{3}}
.\nonumber
\end{eqnarray}

  Let $f(n)=(\frac{n-3}{2})^{2}-2(\frac{3}{4})^{\frac{3}{2}}(3n-8)$, which is monotonically increasing for $n>19$ and $f(19)>0$, then for $n\ge 19$, we have $\mu_3(G_3)=6<\mu_3(G_2)\le \frac{3}{4} \tbinom{4}{3}^{\frac{1}{3}}(3n-8)^{\frac{2}{3}} <(\frac{n-3}{2})^{\frac{4}{3}}<( \lfloor \frac{n-1}{2} \rfloor  \lfloor \frac{n}{2} \rfloor )^{\frac{2}{3}}=\mu_3(G_1)$. For $6\le n\le 18$, we can get the conclusion of our claim by  direct calculation.

For $G'\cong K_1 \vee T_{2}(n-1)$ or  $G'\cong K_3\vee (n-3)K_1$, any new edge can not be added in $G'$, otherwise forming new $3$-clique, a contradiction. Thus we have $G\cong K_1 \vee T_{2}(n-1)$ or  $G\cong K_3\vee (n-3)K_1$, respectively.

For $G'\cong K_5\cup (n-5)K_1$ and $n=6$, we have $G\cong K_{5}\cup K_1$ or $G\cong G_0$, where $G_0$ is obtained by adding an edge between a vertex in $K_5$ and an isolated vertex.
 \end{proof}
 
\setlength{\parindent}{0pt} \textit{\textbf{Acknowledgment:}} The research of the first author is partially supported by the National Natural Science Foundation of China (No.12071097) and the Natural Science Foundation for The Excellent Youth Scholars of the Heilongjiang Province (No.YQ2022A002).

\section*{References}

\bibliography{refer}

\end{spacing}
\end{document}